\documentclass[11pt]{article}

\date{}
\usepackage[utf8]{inputenc} 
\usepackage[T1]{fontenc}    
\usepackage[colorlinks=true,linkcolor=blue,allcolors=blue]{hyperref}       
\usepackage{url}            
\usepackage{booktabs}       
\usepackage{amsfonts}       
\usepackage{nicefrac}       
\usepackage{microtype,nicefrac}      
\usepackage{xspace}
\usepackage{natbib}
\usepackage[toc]{appendix}
\usepackage{minitoc}
\usepackage{subcaption}
\usepackage{float}
\usepackage{comment}
\usepackage{amsmath,amsthm}
\usepackage{graphicx}
\usepackage{rotating}
\usepackage{multirow}

\usepackage{amssymb}
\usepackage{autobreak}
\usepackage{mathtools}
\usepackage[dvipsnames]{xcolor}
\usepackage{bbm}
\usepackage[ruled,vlined, linesnumbered]{algorithm2e}
\usepackage[noend]{algpseudocode}

\usepackage[parfill]{parskip}

\usepackage[toc]{appendix}
\usepackage{minitoc}

\usepackage{bigints}
\usepackage{mathabx}

\makeatletter
\usepackage[colorinlistoftodos]{todonotes}

\newcommand{\innp}[1]{\left\langle #1 \right\rangle}

\newcommand{\cx}{\mathcal{X}}
\newcommand{\cy}{\mathcal{Y}}

\newcommand{\cz}{\mathcal{Z}}

\newcommand{\rr}{\mathbb{R}}
\newcommand{\ee}{\mathbb{E}}

\graphicspath{{imgs/}}

\makeatletter
\def\mathcolor#1#{\@mathcolor{#1}}
\def\@mathcolor#1#2#3{%
  \protect\leavevmode
  \begingroup
    \color#1{#2}#3%
  \endgroup
}
\makeatother

\newcommand*{\vsepfbox}[1]{%
  \begingroup
    \sbox0{\fbox{#1}}%
    \setlength{\fboxrule}{0pt}%
    \mbox{\kern-\fboxsep\fbox{\unhbox0}\kern-\fboxsep}%
  \endgroup
}

\theoremstyle{plain} \numberwithin{equation}{section}
\newtheorem{theorem}{Theorem}[section]
\numberwithin{theorem}{section}
\newtheorem{corollary}[theorem]{Corollary}

\newtheorem{lemma}[theorem]{Lemma}

\theoremstyle{definition}
\newtheorem{definition}[theorem]{Definition}

\newtheorem{example}[theorem]{Example}

\newtheorem{assumption}{Assumption}

\DeclareMathOperator*{\argmin}{argmin}

\makeatletter
\newcommand{\subalign}[1]{%
  \vcenter{%
    \Let@ \restore@math@cr \default@tag
    \baselineskip\fontdimen10 \scriptfont\tw@
    \advance\baselineskip\fontdimen12 \scriptfont\tw@
    \lineskip\thr@@\fontdimen8 \scriptfont\thr@@
    \lineskiplimit\lineskip
    \ialign{\hfil$\m@th\scriptstyle##$&$\m@th\scriptstyle{}##$\hfil\crcr
      #1\crcr
    }%
  }%
}
\makeatother

\newcommand{\R}{\mathbb{R}}

\newcommand\calZ{\mathcal{Z}}
\newcommand\hoeg{\textsc{hoeg+}}

\newcommand{\pa}[1]{\left(#1\right)}

\newcommand{\norm}[1]{\ensuremath{\left\lVert #1 \right\rVert}}

\title{Beyond first-order methods for non-convex non-concave min-max optimization}
\author{%
Abhijeet Vyas\\
  Purdue University\\
  \texttt{vyas26@purdue.edu} 
\and
Brian Bullins\\
  Purdue University\\
  \texttt{bbullins@purdue.edu} \\
}

\begin{document}

\maketitle

\begin{abstract}
We propose a study of structured non-convex non-concave min-max problems which goes beyond standard first-order approaches. Inspired by the tight understanding established in recent works \citep{adil2022optimal, lin2022perseus}, we develop a suite of higher-order methods which show the improvements attainable beyond the monotone and Minty condition settings. Specifically, we provide a new understanding of the use of discrete-time $p^{th}$-order methods for operator norm minimization in the min-max setting, establishing an $O(1/\epsilon^\frac{2}{p})$ rate to achieve $\epsilon$-approximate stationarity, under the weakened Minty variational inequality condition of \cite{diakonikolas2021efficient}. We further present a continuous-time analysis alongside rates which match those for the discrete-time setting, and our empirical results highlight the practical benefits of our approach over first-order methods.
\end{abstract}

\section{Introduction}
In this work we study the classic min-max problem:
\begin{equation}\label{eq:minmax}
\min_{x\in\cx}\max_{y \in \cy} f(x,y),
\end{equation}
whereby $f : \cx \times \cy \mapsto \R$ may be non-convex in $x$ and non-concave in $y$, and where we assume $f$ is smooth (up to various orders) in both $x$ and $y$. Problems of this form naturally arise in many different areas of machine learning, from training generative adversarial networks \citep{goodfellow2020generative}, to adversarial training/robustness \citep{madry2018towards, zhang2019theoretically}, along with more general robust optimization objectives \citep{ben2009robust}.

In the case where $f$ is smooth and convex-concave in $x$ and $y$, respectively, algorithms such as the extragradient method \citep{korpelevich1976extragradient} and its generalizations such as Mirror Prox \citep{nemirovski2004prox} are able to converge (in terms of duality gap) at a rate of $O(1/\epsilon)$, and this is known to be tight for first-order methods \citep{ouyang2021lower}. However, as the complexity (and inherent non-convexity) of the underlying models for these large-scale problems increases, so too does the need to expand \emph{beyond} the convex-concave setting.

Unfortunately, in the most general constrained non-convex non-concave setting, it would appear there is not much hope for efficiently finding a stationary point, as even doing so approximately has been shown to be FNP-complete \citep{daskalakis2021complexity}. Recently, however, there has been significant interest in overcoming these difficulties by looking instead at certain \emph{structured} non-convex non-concave problems. Specifically, we take inspiration from the previous work by \cite{diakonikolas2021efficient}, which, for the more general case of variational inequalities, establishes a useful characterization of problems defined in terms of a weakening of the standard Minty condition. \cite{diakonikolas2021efficient} further provide a generalization of the extragradient method, which they show manages to reach an $\epsilon$-approximate stationary point (w.r.t. the norm of the operator) at a $O(1/\epsilon^2)$ rate. In addition, by choosing the operator $F = [\nabla_x f, -\nabla_y f]^\top$, they establish the same rate for reaching stationary points of the unconstrained min-max optimization problem, under their \textit{weak}-MVI condition (which is weaker than assuming the variational inequality satisfies the Minty condition \citep{song2020optimistic, lin2022perseus}), and they further show how their results may extend to non-Euclidean settings.

Building on these efforts, we aim to understand the opportunities afforded by going \emph{beyond} the standard first-order extragradient approach. In particular, recent works \citep{adil2022optimal,lin2022perseus} have presented optimal $p^{th}$-order methods for approximately solving \emph{monotone} variational inequalities. In this work, we show how to relax this monotonicity condition, in a similar manner to that of \cite{diakonikolas2021efficient}, that is suitable to a variant of the higher-order extragradient method.

\subsection{Our Contributions}

Our main contributions are as follows.

\begin{itemize}
    \item We propose the \textsc{hoeg+} method, based on a higher-order variant \citep{adil2022optimal} of the extragradient algorithm \citep{korpelevich1976extragradient}, and we establish convergence results for structured non-convex non-concave min-max problems. 
    \item For a generalized \textit{weak}-MVI condition inspired by \citep{diakonikolas2021efficient}, we show that the $p^{th}$-order instance of our algorithm finds $\epsilon$-approximate stationary points at a $O(1/\epsilon^\frac{2}{p})$ rate. Furthermore, these are to our knowledge the first results that go beyond the $O(1/\epsilon^2)$ rate for the \textit{weak}-MVI setting.
    \item In the continuous-time regime, we show that the algorithm achieves an analogous rate of $\min_{0\leq t \leq T}\|F(z_t)\| \leq O(1/T^\frac{p}{2})$ under the \textit{weak}-MVI condition. Furthermore, under the co-monotonicity condition on the operator which implies \textit{weak}-MVI, the first-order instance of the algorithm has $\|F(z(t))\|$ decreasing and thus $\|F(z(t))\| \leq O(1/\sqrt{t})$.
    \item We provide a study of the empirical performance of the first- and second-order instances of our algorithm based on \textit{weak}-MVI examples with the standard min-max operator $F = (\nabla_x f, -\nabla_y f)$, as well as the competitive operator $F_{\alpha}$ introduced in \citep{vyas2022competitive}. 
\end{itemize}

\begin{table}[H]\label{tab:discrete}
\centering
	\begin{tabular}{ l c c c c}
	\toprule
		Algorithm \\(Reference)  & \multicolumn{3}{c}{Bounded range ($\alpha \leq f(x) \leq \beta\ \forall x \in \mathbb{R}^n$)} & General operator?\\
  \cmidrule(lr{0.2em}){2-4}
   & $p=1$ & $p = 2$ & $p > 2$ & \\
		\midrule
		Gradient Descent\\(Folklore) & $O\pa{1/\epsilon^2}$ & --- & --- & No \\
  \midrule
  		Cubic Regularization\\\citep{nesterov2006cubic} & --- & $O\pa{1/\epsilon^{\frac{3}{2}}}$ & --- & No \\
		\midrule
    		AR$p$ \citep{birgin2017worst} & $O\pa{1/\epsilon^2}$ & $O\pa{1/\epsilon^{\frac{3}{2}}}$ & $O\pa{1/\epsilon^{\frac{p+1}{p}}}$ & No \\
		\toprule
  	\toprule
		 & \multicolumn{3}{c}{$p^{th}$-Order \textit{weak}-MVI (Assumption \ref{assump:pwmvi})} & \\
  \cmidrule(lr{0.2em}){2-4}
    & $p=1$ & $p = 2$ & $p > 2$ & \\
		\midrule
		Extragradient+\\\citep{diakonikolas2021efficient} & $O(1/\epsilon^2)$ & --- & ---  & Yes\\
		\midrule
		\textsc{hoeg+} (Algorithm \ref{alg:mainalg})\\ Theorem \ref{theorem:balanced} \textbf{(This Paper)}  & $O(1/\epsilon^2)$ & $O(1/\epsilon)$ & $O(1/\epsilon^\frac{2}{p})$  & Yes\\
		\toprule
	\end{tabular}
 \caption{Discrete-time convergence rates for reaching $\epsilon$-approximate stationary points in the unconstrained Euclidean setting, under bounded range and $p^{th}$-order \textit{weak}-MVI assumptions, respectively, in addition to $p^{th}$-order smoothness. Note that for $p=1$, the \textit{weak}-MVI assumption leads to the same rate as in the standard (bounded range) non-convex setting, whereas our method highlights the improvements attainable when going beyond first-order methods, for $p > 1$.}
\end{table}

\begin{table}[H]
\centering
	\begin{tabular}{ l c c c}
  	\toprule
		 & Monotone & $p^{th}$-Order \textit{weak}-MVI (Assumption \ref{assump:pwmvi}) \\
		\midrule
		Algorithm 2 in \citep{lin2022continuous}\\Theorem 4.1& $O(1/\epsilon^\frac{2}{p})$ & --- \\
		\midrule
		\textsc{hoeg+}(Algorithm \ref{alg:mainalg})\\Theorem \ref{thm:continuoustime} \textbf{(This Paper)}  & $O(1/\epsilon^\frac{2}{p})$ & $O(1/\epsilon^\frac{2}{p})$ \\
		\toprule
	\end{tabular}
 \caption{Continuous-time convergence rates for reaching $\epsilon$-approximate stationary points in the unconstrained Euclidean setting, under monotonicity and $p^{th}$-order \textit{weak}-MVI assumptions, in addition to $p^{th}$-order smoothness.}
\end{table}

\subsection{Additional Related Works}

Recently, several works \citep{bullins2022higher, jiang2022generalized,adil2022optimal, lin2022perseus} have shown how to achieve $\epsilon$-approximate weak solutions to monotone variational inequalities under $p^{th}$-order oracle access, wherein the $O(1/\epsilon^\frac{2}{p+1})$ rates of \cite{adil2022optimal, lin2022perseus} (which remove additional logarithmic factors) are optimal. However, these works have focused predominantly on the convex-concave (monotone operator) setting, though \cite{lin2022perseus} also show a rate of $O(1/\epsilon^\frac{2}{p})$ under the Minty condition. In addition, \cite{lin2022continuous} provide a rate of $O(1/\epsilon^\frac{2}{p})$ to reach a point with small operator norm in the monotone setting.

In more typical non-convex optimization settings, especially given the ubiquity of large-scale models for modern machine learning, much attention has been paid to developing methods which can find approximate stationary points (that is, places where $\norm{\nabla f(x)} \leq \epsilon$). One such example is the cubic regularization method \citep{nesterov2006cubic} and its higher-order generalization \citep{birgin2017worst}, providing rates that match the lower bounds \citep{carmon2020lower}. Furthermore, even first-order (as well as Hessian-vector product-based) methods have also been shown to benefit from higher-order smoothness, with certain accelerated methods giving rise to even faster rates \citep{agarwal2017finding, carmon2017convex, carmon2018accelerated, jin2018accelerated, li2022restarted}, though there remains a small gap between upper and lower bounds \citep{carmon2021lower}.

The continuous-time regime has proven to be effective in analyzing the performance of algorithms for minimization problems \citep{latz2021analysis,wilson2016lyapunov,wibisono2016variational,shi2021understanding}, min-max optimization problems \citep{lin2022continuous,vyas2022competitive}, and in the study of continuous games \citep{mazumdar2020gradient}. \cite{latz2021analysis} analyzes stochastic gradient descent in the continuous-time regime while \cite{wilson2016lyapunov,wibisono2016variational,shi2021understanding} provide a continuous-time perspective of the accelerated variants of gradient descent. \cite{lin2022continuous} study the continuous-time version of the dual-extrapolation algorithm, and \cite{vyas2022competitive} build on the competitive gradient descent algorithm \citep{schafer2019competitive} to design a new min-max algorithm based on re-scaling the cross-terms of the Jacobian of the operator.
\section{Preliminaries}

In this section we discuss the notations and key assumptions that formulate the setting for our algorithm.
We start by defining the approximate and exact stationary points of an operator $F$. 

\begin{definition}[Stationary points]
A point $z\in \cz\subseteq\rr^d$, is an $\epsilon$-approximate stationary point if
$$\|F(z)\|\leq \epsilon,$$ and it is an exact stationary point if
$$\|F(z)\|=0.$$
\end{definition}

We then define the solution set to the Stampachchia variational inequality for any field $F$.
\begin{definition}[Solution set 
$\cz^*$]
We refer to the solutions of the Stampachchia Variational Inequality (SVI):

$$\langle F(z^*),z-z^*\rangle\geq 0 $$
 as the set $\cz^* \subseteq \cz$. 
\end{definition}
To solve the problem in Eq.~\eqref{eq:minmax} we consider the field $F =  \begin{bmatrix}
\nabla_x f\\
-\nabla_y f 
\end{bmatrix}$, derived from the function $f$. For the unconstrained setting, i.e, $\cz=\rr^d$, we have $\|F(z^*)\|=0~\forall z^* \in \cz^*$. Furthermore all stationary points of Eq.\eqref{eq:minmax} satisfy the SVI. We assume $\cz^*\neq \phi$.

We start by defining the monotonicity of an operator. 

\begin{definition}[Monotonicity]\label{def:monotonicity}
An operator is monotone if for all $z_a,z_b\in \cz$,
$$\langle F(z_a)-F(z_b),z_a-z_b \rangle \geq 0$$
\end{definition}

Standard examples of monotone operators include the gradient $\nabla f$ of a convex function $f(x)$ and the concatenated gradient $(\nabla_x f, -\nabla_y f)$ for a convex-concave $f(x,y)$ function. 

We now present the definition of comonotonicity condition which generalizes monotonicity. This was used as the key non-monotone condition to provide first-order algorithms in \cite{lee2021fast}.  

\begin{definition}[$\rho$-comonotone]
An operator $F$ is $\rho$-comonotone if,
$$\langle F(z_a)-F(z_b),z_a-z_b \rangle > \rho \|F(z_a)-F(z_b)\| ~\forall z_a,z_b \in \cz$$
\end{definition}
Note that for $\rho$-comonotonicity implies monotonicity for $\rho\geq 0$.

Inspired by the \textit{weak}-MVI condition in \cite{diakonikolas2021efficient} we generalise to the \textit{weak}-MVI condition to $p$-dimensions. This condition is the key condition under which we prove the convergence of our algorithm.

\begin{assumption}[$p^{th}$-Order \textit{weak}-\textsc{mvi}]\label{assump:pwmvi}
There exists $z^* \in \cz^*$ such that:
\begin{equation}\label{assmpt:balanced}\tag{\textsc{a}$_1$}
    (\forall z \in \rr^d):\quad 
    \innp{F(z), z - z^*} \geq -\frac{\rho}{2} \|F(z)\|^{\frac{p+1}{p}},
\end{equation}
for some parameter $\rho\geq0$. 
\end{assumption}

For $\rho=0$ the condition is the well-known MVI condition \cite{mertikopoulos2018optimistic}. Furthermore for $p=1$, $\rho$-comonotonicity implies $-\frac{\rho}{2}$ \textit{weak}-MVI. Overall, monotonicity implies MVI which implies \textit{weak}-MVI. 

Finally we define the smoothness of an operator.
\begin{assumption}[$p^{th}$-Order Smoothness]
\begin{equation}\label{assmpt:smooth}\tag{\textsc{a}$_2$}
    \|F(z_b)-\tau_{p-1}(z_b,z_a)\|\leq \frac{L_p}{p!}\|z_b-z_a\|^p,\forall z_a,z_b \in \cz
\end{equation}
\end{assumption}

We now define some quantities used in the algorithm.
\begin{definition}
We define $\tau_p$ as the Taylor approximation of $F$ at $z_b$ centered at $z_a$,
\begin{equation}\label{eq:tau-def}
    \tau_p(z_a,z_b) :=  \sum_{i=0}^p \nabla^i F(z_a)[z_b-z_a]^i
\end{equation}
\end{definition}

\begin{definition}
We define $\Phi_p$ as the regularized Taylor approximation,
\begin{align}\label{eq:phi-def}
    &\Phi_p(z_b,z_a) :=   \tau_{p-1}(z_b,z_a) +\frac{2L_p}{p!}\|z_b-z_a\|^{p-1}(z_b-z_a).
\end{align}
\end{definition}
\section{Convergence Results}

In this section we will present our algorithm and discuss in detail the conditions under which it converges to stationary points. we will start by analysing the behaviour of the algorithm under the limit of the learning rate approaching zero. Under the said limit the algorithm can be represented as a dynamical system, and we will analyze the system which is the continuous-time analogue of our algorithm with the conditions of weak-MVI and comonotonicity on the operator. We will then proceed to analyze the discrete-time algorithm under the generalized weak-MVI condition.  

We now present our algorithm \hoeg, a higher-order variant of the extra-gradient method. The algorithm is as follows.

\begin{algorithm}[H]
  $z_0 \in \cx \times\cy$;\\
  \For{k = 0 to k = K}{
        $z_{k+\frac{1}{2}} = z'\ \textrm{s.t.}\ \Phi(z',z_k)=0$\\
    $\lambda_k = \frac{1}{2}\|z_{k+\frac{1}{2}} - z_k\|^{1-p}$
    \\
    $z_{k+1} = \argmin_{z'' \in \rr^d} \Big\{ \innp{F(z_{k+\frac{1}{2}}), z'' - z_{k+\frac{1}{2}}} + \frac{L_p}{p!\lambda_k}\|z_{k}'' - z_k\|^2\Big\}$= $z_k - \frac{p!\lambda_k}{2L_p} F(z_{k+\frac{1}{2}})$
  }
      \Return $z_{out} = \argmin_{z_{k+\frac{1}{2}} , 0\leq k\leq K}\|z_{k+\frac{1}{2}}\|$ \caption{\textsc{hoeg+}}\label{alg:mainalg}
\end{algorithm}

We now discuss the convergence of the dynamics of the system of differential equations obtained in the limit of the higher-order extra-gradient method. The continuous-time dynamics of the $p^{th}$-order dual extrapolation algorithm \citep{nesterov2007dual}, as given by \cite{lin2022continuous}, are

\begin{equation}\label{sys:DE}
\begin{array}{lll}
\dot{s}(t) = - \tfrac{F(z(t))}{\|F(z(t))\|^{1-1/p}}, & v(t) = z_0 + s(t), & z(t) - v(t) + \tfrac{F(z(t))}{\|F(z(t))\|^{1-1/p}} = \textbf{0}, 
\end{array} 
\end{equation}

While \cite{lin2022continuous} shows that \eqref{sys:DE} is the dynamical system for dual extrapolation, we know that dual extrapolation is equivalent to the extra-gradient method for the unconstrained Euclidean setting \citep{nesterov2007dual}, and so the system of \eqref{sys:DE} is the continuous-time analog of \hoeg.

\begin{theorem}\label{thm:continuoustime}
    Let the operator $F$ satisfy~\eqref{assmpt:balanced} with $\rho < 2$. Then for the continuous-time dynamics~\eqref{sys:DE} we have $\min_{0 \leq s\leq t}  \|F(z(s))\|^2$ follows $O(\frac{1}{t^p})$.
\end{theorem}
\begin{proof}

Let us consider the lyapunov function $\mathcal{E} = \|s(t)\|^2$.

Expanding this equality with any $z \in \cz $, we have
\begin{equation*}\label{eq:lyapunov-derivative}
\frac{d\mathcal{E}(t)}{dt} = \frac{2(\langle F(z(t)), z_0 - z\rangle + \langle F(z(t)), z - z(t)\rangle + \langle F(z(t)), z(t) - v(t)\rangle)}{\|F(z(t))\|^{1 - 1/p}}
\end{equation*}
Since $\dot{s}(t) = - \tfrac{F(z(t))}{\|F(z(t))\|^{1-1/p}}$, we have
\begin{equation*}
\frac{\langle F(z(t)), z_0 - z\rangle}{\|F(z(t))\|^{1 - 1/p}} = -\langle \dot{s}(t), z_0 - z\rangle. 
\end{equation*}
Since $F$ satisfies the weak-MVI assumption we have,
$$\langle F(z(t)),z-z(t)\rangle \leq \frac{\rho}{2}\|F(z(t))\|^\frac{p+1}{p}-\langle F(z(t)),z^*-z \rangle$$

Since $z(t) - v(t) + \tfrac{F(z(t))}{\|F(z(t))\|^{1-1/p}} = \textbf{0}$, we have
\begin{equation*}
\frac{\langle F(z(t)), z(t) - v(t)\rangle}{\|F(z(t))\|^{1 - 1/p}} = -\|F(z(t))\|^{\frac{2}{p}}. 
\end{equation*}
Plugging these pieces together in Eq.~\eqref{eq:lyapunov-derivative} yields that, for any $z \in \cz$, we have
\begin{align}\label{inequality:nonasymptotic-second}
\frac{d\mathcal{E}(t)}{dt} &\leq 2\langle \dot{s}(t), z-z_0\rangle + 2\langle \dot s, z^* - z\rangle - (2-\rho)\|F(z(t))\|^{\frac{2}{p}}.\\
&\leq 2\langle \dot{s}(t), z^*-z_0 \rangle - (2-\rho)\|F(z(t))\|^{\frac{2}{p}}. 
\end{align}

Integrating this inequality over $[0, t]$, using $\mathcal{E}(0) = 0$ and observing $\langle s(t), z^*-z_0\rangle - \|s(t)\|^2 \leq \frac{1}{4}\|z^*-z_0\|^2$, $\|z_0 - z^*\| \leq D$, we have,
\begin{equation*}
\int_0^t \|F(z(s))\|^{\frac{2}{p}} \; ds \leq \frac{1}{2-\rho}\left(\langle s(t), z^*-z_0\rangle - \|s(t)\|^2\right) \leq \frac{D^2}{4(2-\rho)}, \quad \textnormal{for all } t \geq 0.  
\end{equation*}

Now let $m= \min_{0 \leq s\leq t}  \|F(z(s))\|^{\frac{2}{p}} = (\min_{0 \leq s\leq t}  \|F(z(s))\|)^{\frac{2}{p}}$, where the second equality is true since $\|.\|^\frac{2}{p}$ is increasing for $p\geq 1$. 
Then we have,
\begin{align*}
mt \leq \int_0^t \|F(z(s))\|^{\frac{2}{p}} \; ds \leq \frac{D^2}{4(2-\rho)},   \rightarrow
\min_{0 \leq s\leq t}  \|F(z(s))\|^2 \leq \frac{D^{2p}}{4(2-\rho)^pt^p}=O(\frac{1}{t^p}) \quad \textnormal{for all } t\geq 0.
\end{align*}
\end{proof}
  
While the weak-MVI condition with $\rho<2$ is sufficient to obtain the desired $O(\frac{1}{t^p})$ rates, we show that for an operator that is comonotone with $\rho>-1$ the dynamics \eqref{sys:DE} are such that $\|F(z(t))\|$ is decreasing furthermore since comonotonicity with $\rho>-1$ implies weak-MVI with $\rho>2$ we have the conditions of Theorem \ref{thm:continuoustime} satisfied and thus $\|F(z(t))
\| = \min_{0\leq s \leq t}\|F(x(s))\| \leq O(\frac{1}{t})$.

\begin{corollary}
    If the operator $F$ is $\rho$-comonotone with $\rho >-1$, for the time dynamics~\eqref{sys:DE} with $p=1$ we have $\|F(z(t))\|$ is decreasing and $\|F(z(t))\|^2$ follows $O(\frac{1}{t})$.
\end{corollary}
\begin{proof}
Since the operator is $\rho$-comonotone with $\rho >-1$ we have
$$\langle F(z_1)-F(z_2),z_1-z_2 \rangle > -\|F(z_1)-F(z_2)\| ~\forall z_1,z_2 \in \cz$$
Setting $z_2=z^*$ where $\|F(z^*)\|=0$, we obtain that $F$ satsifies assumption \eqref{assmpt:balanced} with $\rho < 2$ and thus from Theorem \ref{thm:continuoustime} we have,
\begin{align}\label{eqn:min}
    \min_{0 \leq s\leq t}  \|F(z(s))\|^2 \leq O(\frac{1}{t})
\end{align}
Furthermore setting $z_1 = z +\tau \delta z,~z_2 = z$ and taking limit $\tau \rightarrow 0$ we obtain
$$\lim_{\tau \rightarrow 0} \langle F(z +\tau \delta z)-F(z),z +\tau \delta z-z \rangle > - \lim_{\tau \rightarrow 0}\|F(za +\tau \delta z)-F(z)\|^2$$
this gives
$$\lim_{\tau \rightarrow 0} \langle \frac{F(z +\tau \delta z)-F(z)}{\tau},\delta z \rangle > - \lim_{\tau \rightarrow 0}\|(\frac{F(z +\tau \delta z)-F(z)}{\tau})^2\|$$
which gives
\begin{align}\label{eqn:operator-eigen}
\langle \nabla F(z) \delta z,\delta z \rangle + \|\nabla F(z) \delta z\|^2 \geq 0
\end{align}
We now prove that  $t \mapsto \|F(z(t))\|$ is non-increasing for the dynamics ~\eqref{sys:DE}. 
The dynamics are,
\begin{equation*}
\dot{s}(t) = -F(z(t)), \quad v(t) = z_0 + s(t), \quad z(t) - v(t) + F(z(t)) = \textbf{0}.
\end{equation*}
In this case, we can write $z(t) = (I + F)^{-1}v(t)$. Since $v(\cdot)$ is continuously differentiable, we have $x(\cdot)$ is also continuously differentiable. Define the function $g(t) = \frac{1}{2}\|v(t) - z(t)\|^2$. 
Then we have, 
\begin{equation*}
\dot{g}(t) = \langle \dot{v}(t) - \dot{z}(t), v(t) - z(t) \rangle = -\langle \dot{v}(t) - \dot{z}(t), \dot{v}(t) \rangle = -\|\dot{v}(t) - \dot{z}(t)\|^2 - \langle \dot{v}(t) - \dot{z}(t), \dot{z}(t) \rangle. 
\end{equation*}
Now
$$\dot{v}(t) - \dot{z}(t) = \nabla F(z(t)) \dot{z}(t)$$
thus choosing $z = z(t), \delta z = \dot z (t)$ in Eq.~\eqref{eqn:operator-eigen} we obtain,
\begin{align*}
 \dot{g}(t) = -(\|\nabla F (z(t)) \dot z(t) \|^2+\langle \nabla F(z(t)) \dot z,\dot z \rangle) < 0
\end{align*}
Thus $\|F(z(t))\|$ is decreasing and combined with \eqref{eqn:min} we have, 
$$ \|F(z(t))\|^2 \leq O(\frac{1}{t})$$ which is the statement of the corollary.
\end{proof}

 Note that $-1$-comontoncitiy is sufficient for $\|F(z(t))\|$ to be non-increasing but violates the conditions of Theorem \ref{thm:continuoustime}.
 
\subsection{Discrete Time}

For discrete time we discuss the results in three different settings discussed above. We start by establishing several key lemmas, the first of which comes from \cite{adil2022optimal}.

\begin{lemma}[\cite{adil2022optimal}]\label{lemma:supportive}
For the iterates of the Algorithm \ref{alg:mainalg} $\{z_k\}_{k=1}^K$ and an SVI solutio $z^*\in \cz^*$ we have

\begin{align}\label{eq:main}
    \sum_{i=0}^K \lambda_k \frac{p!}{L_p}\langle F(z_{k+\frac{1}{2}}),z_{k+\frac{1}{2}}-z \rangle &\leq \|z^*-z_0\|^2-\frac{15}{16}\sum_{k=0}^K\|z_k-z_{k+\frac{1}{2}}\|^2
\end{align}
\end{lemma}
We provide the proof of this lemma for completeness in Appendix \ref{subsec:supportive-dt}

\begin{lemma}\label{lemma:upper-b}
For $z_{k+\frac{1}{2}},z_k$ obtained from Algorithm \ref{alg:mainalg} we have
\begin{equation}\label{upper-b}
    \|F(z_{k+\frac{1}{2}})\|\leq \frac{3L_p}{p!}\|z_{k+\frac{1}{2}}-z_k\|^p
\end{equation}
\end{lemma}

\begin{proof}
From \eqref{assmpt:smooth} and noting that for any two vectors $a\in \rr^d,b\in\rr^d$ we have,
\begin{equation}
    \|a\|-\|b\|\leq \|a-b\|
\end{equation}
we have
\begin{equation}
    \|F(z_{k+\frac{1}{2}})\|\leq \tau_{p-1}(z_{k+\frac{1}{2}},z_k)+\frac{L_p}{p!}\|z_{k+\frac{1}{2}}-z_k\|^p
\end{equation}
Combining with the update rule of Algorithm \ref{alg:mainalg}, we have $\|\tau_{p-1}(z_{k+\frac{1}{2}},z_k)\| = \frac{2L_p}{p!}\|z_{k+\frac{1}{2}}-z_k\|^p$ which proves the statement of the lemma.
\end{proof}

We now prove convergence for our Algorithm \ref{alg:mainalg} with a $p^{th}$-order weak-MVI condition where $p$ is closely tied to the order of the instance of Algorithm \ref{alg:mainalg} used.

\begin{theorem}\label{theorem:balanced}
For $F$ satisfying ~\eqref{assmpt:balanced} with $\rho\leq\frac{15}{16}(\frac{p!}{L_p})^{\frac{p+1}{p}}$ and running Algorithm~\ref{alg:mainalg} we have,
\begin{itemize}
    \item[(i)] for all $k\geq 1:$ 
    $$
        \frac{1}{K+1}\sum_{k=0}^K \|F(z_{k+\frac{1}{2}})\|^\frac{2}{p} \leq \frac{ \|z_0 - z^*\|^2}{c_2(K+1)}.
    $$
    In particular, we have that 
    \begin{align}
        \min_{0\leq k\leq K} \|F(z_{k+\frac{1}{2}})\|^\frac{2}{p} \leq \frac{\|z_0 - z^*\|^2}{c_2(K+1)}\rightarrow 
        \min_{0\leq k\leq K} \|F(z_{k+\frac{1}{2}})\|^2 \leq \frac{C}{(K+1)^p}  \nonumber
    \end{align}
    where $C=\frac{\|z_0-z^*\|^{2p}}{c_2^p}$ and 
    \begin{align}
        \ee_{k \leq \mathrm{Unif}\{0, \dots, K\}}\big[\|F(z_{k+\frac{1}{2}})\|^2\big]\leq (\frac{ \|z_0 - z^*\|^2}{c_2(K+1)})^p\leq \frac{C}{(K+1)^p},
    \end{align}
    
    where $k \leq \mathrm{Unif}\{0, \dots, K\}$ denotes an index $i$ chosen uniformly from $\{0, \dots, K\}.$

    \item[(ii)] for $K = O(\frac{1}{\epsilon^\frac{p}{2}})$ number of iterations the output $z_{out}$ of Algorithm \ref{alg:mainalg} is an $\epsilon$-approximate stationary point i.e., it satisfies $\|F(z_{out})\| \leq \epsilon$
\end{itemize}
\end{theorem}
\begin{proof}
Setting $z=z^*$ in Lemma 3.2 \citep{adil2022optimal}, for Algorithm \ref{alg:mainalg},  we have: 
\begin{align*}
    \sum_{i=0}^K \lambda_k \frac{p!}{L_p}\langle F(z_{k+\frac{1}{2}}),z_{k+\frac{1}{2}}-z^* \rangle \leq \|z_0 - z^*\|^2-\frac{15}{16}\sum_{k=0}^K\|z_k-z_{k+\frac{1}{2}}\|^2,\nonumber
\end{align*}
Adding $\sum_{k=1}^K\rho\lambda_k\|F(z_{k+\frac{1}{2}})\|^{\frac{p+1}{p}}$ 
to both sides of Eq. \eqref{eq:main}, and noting from Eq.~\eqref{lemma:upper-b} that,
\begin{align}
     \rho\lambda_k\|F(z_{k+\frac{1}{2}})\|^{\frac{p+1}{p}} = \rho\frac{\|F(z_{k+\frac{1}{2}})\|^{\frac{p+1}{p}}}{2\|z_k-z_{k+\frac{1}{2}}\|^{p-1}}\leq \rho (\frac{L_p}{p!})^{\frac{p+1}{p}}\|z_k-z_{k+\frac{1}{2}}\|^2\nonumber
\end{align}
we obtain
\begin{align}\label{main2}
    \sum_{i=0}^K \lambda_k (\frac{p!}{L_p}\langle F(z_{k+\frac{1}{2}}),z_{k+\frac{1}{2}}-z^* \rangle + \rho\|F(z_{k+\frac{1}{2}})\|^{\frac{p+1}{p}}) \\
    \leq \|z_0 - z^*\|^2-(\frac{15}{16}-\rho(\frac{L_p}{p!})^{\frac{p+1}{p}})\sum_{k=0}^K\|z_k-z_{k+\frac{1}{2}}\|^2\nonumber
\end{align}
Now we set $c_1=\frac{15}{16}-\rho(\frac{L_p}{p!})^{\frac{p+1}{p}}$. Upon choosing $\rho$ such that $c_1 > 0$ we note by assumption \eqref{assmpt:balanced} that the LHS in Eq.~\eqref{main3} is non negative.
Setting $ c_2 = (\frac{p!}{3L_p})^\frac{2}{p}$ we observe from Lemma~\ref{lemma:upper-b} that
$$ (\frac{p!}{3L_p})^\frac{2}{p} \|F(z_{k+\frac{1}{2}})\|^\frac{2}{p} \leq \|z_k-z_{k+\frac{1}{2}}\|^2$$
Combining the above we have for the $\frac{1}{2}$-step iterates of the Algorithm \ref{alg:mainalg},
\begin{align}\label{final1}
    \frac{c_2}{c_1}\sum_{k=0}^K\|F(z_{k+\frac{1}{2}})\|^{\frac{2}{p}}&\leq \sum_{k=0}^K\|z_k-z_{k+\frac{1}{2}}\|^2\leq \frac{\|z_0 - z^*\|^2}{c_1}\nonumber
\end{align}
This implies the statement of the theorem.
\end{proof}

We also prove convergence for the $p^{th}$-order instance of the Algorithm \ref{alg:mainalg} under a $q^{th}$ order weak-MVI condition which is de-coupled from the order $p$. This allows us for instance to obtain guarantees for higher order methods when run on the original weak-MVI condition with $p=1$ as studied in the experiments of Section \ref{sec:experiments}. The proof is provided in \ref{theorem:imbalanced}.

 As a corollary we show that, for the monotone setting, our algorithm recovers a rate of $O(\frac{1}{k^p})$ on $\|F(x_k)\|^2$, as also obtained by \cite{lin2022continuous}.

\begin{corollary}
The $\frac{1}{2}$ step iterates of Algorithm \ref{alg:mainalg}, $z_{k+\frac{1}{2}}$, achieve $O(1/k^p)$ rate for monotone $F$.
\end{corollary}

\begin{proof}
From Eq.\eqref{eq:main} and using monotonicity we have,
\begin{align*}
    0 \leq \sum_{i=0}^K \lambda_k \frac{p!}{L_p}\langle F(z_{k+\frac{1}{2}}),z_{k+\frac{1}{2}}-z^* \rangle\leq \|z^*-z_0\|^2-\frac{15}{16}\sum_{k=0}^K\|z_k-z_{k+\frac{1}{2}}\|^2,\nonumber
\end{align*}
This gives
\begin{align*}
    \frac{16}{15}(\frac{p!}{3L_p})^\frac{2}{p} \sum_{i=0}^K\|F(z_{k+\frac{1}{2}})\|^{\frac{2}{p}}\leq \sum_{i=0}^K\|z_k-z_{k+\frac{1}{2}}\|^2\leq \frac{16}{15} \|z_0 - z^*\|^2
\end{align*}

and thus we have,
\begin{equation}
    \begin{gathered}
    \min_{0\leq k\leq K} \|F(z_{k+\frac{1}{2}})\|^\frac{2}{p} \leq \frac{\|z_0 - z^*\|^2}{c_2(K+1)}\rightarrow 
    \min_{0\leq k\leq K} \|F(z_{k+\frac{1}{2}})\|^2 \leq \frac{C}{(K+1)^p} 
    \end{gathered}
\end{equation}
\end{proof}
Note that for the monotone case \cite{yoon2021accelerated} obtain faster rates for the $p=1$ case by using the anchoring technique used by \cite{diakonikolas2020halpern} to obtain the same rate for a subclass of monotone problems, in particular they obtain $\|F(x_K)\|^2 \leq O(\frac{1}{K^{2}})$ while our algorithm guarantees a rate of $O(\frac{1}{K})$. Thus $\min_{0\leq k\leq K}\|F(x_k)\|^2 \leq O(\frac{1}{K^p})$ is not tight and obtaining faster rates for higher-order methods in the monotone setting remains an open problem. 
\section{Experiments}\label{sec:experiments}

We now illustrate the empirical performance\footnote{The code for all the experiments can be found at \href{https://github.com/AbhijeetiitmVyas/Higher-Order-Methods-for-Weak-MVIs}{Link to code}} of our method on different examples.

\subsection{With the standard min-max operator}

While \cite{diakonikolas2021efficient} give an example of a \textit{weak}-MVI function in the simplex-constrained setting, our analysis does not assume the simplex setting and thus we provide experiments on a modified version of the example "Forsaken" introduced in \citet{pethick2022escaping} to obtain a \textit{weak}-MVI function in the Euclidean setting. Note that our \textit{weak}-MVI condition on $\rho$ for $p=1$ is slightly different from that of \cite{diakonikolas2021efficient}.

\begin{example}
\begin{equation}\label{example:mforsaken}\tag{Modified-Forsaken}
    \min_{|x|\leq 2} \max_{|y|\leq 2} f(x,y) = x(y-1.5)+h(x)-h(y)
\end{equation}
where $h(t) = \frac{t^2}{4}-\frac{t^4}{2}+\frac{t^6}{6}$.
\end{example}

The only real-valued stationary point of the above function is $z^* = (1.31147,1.47596)$ and we numerically verify that this is a weak-MVI solution to the problem. We illustrate the performance of the $1^{st}$ and $2^{nd}$ order instances of our algorithm on this example in Figure \eqref{figure:mforsaken}. This corresponds to using a $p=2$ and $p=1$ method on a $p=1$ order \textit{weak}-MVI condition

Note that while Theorem~\ref{theorem:balanced} shows convergence with operator norm decreasing at a rate of $O(\frac{1}{n^p})$ for the $p^{th}$-order instance of our algorithm on the $p^{th}$-order \textit{weak}-MVI condition, we show convergence for some $p^{th}$-order instance on a $q^{th}$ order \textit{weak}-MVI condition at a rate of $O(\frac{1}{n^p})$ for the operator norm in appendix section \ref{theorem:imbalanced}

\begin{figure}
\centering
\subcaptionbox{The iterates of the algorithm\label{figure:cgdxy}}
[.48 \textwidth]{\includegraphics[width=1\linewidth]{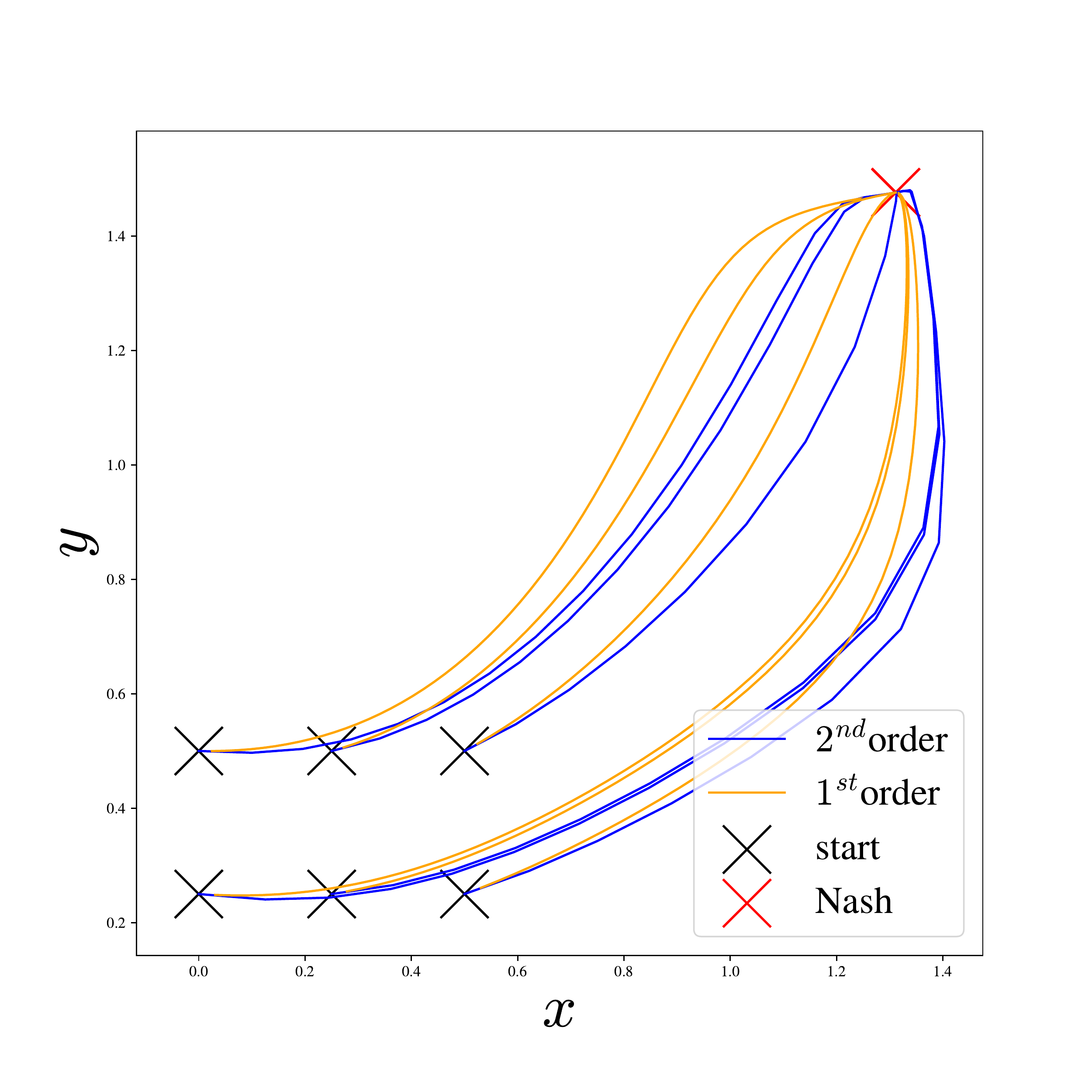}}
\subcaptionbox{The operator norm $\|F\|;$ \label{figure:cgoxy}}
[.48 \textwidth]{\includegraphics[width=1\linewidth]{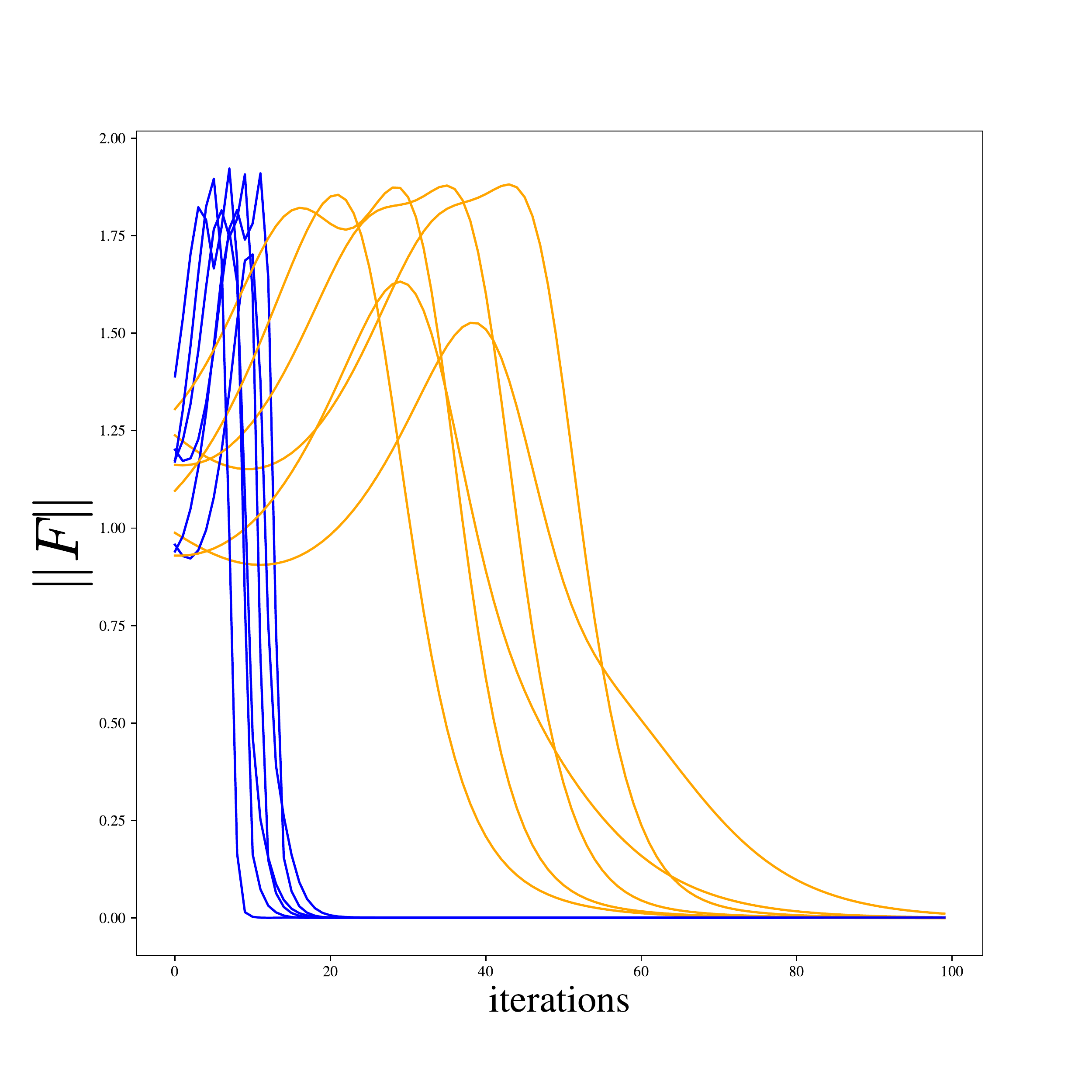}}
\caption{1st and 2nd Order method for the \textit{weak}-MVI example \eqref{example:mforsaken}. For the 1st order method we use $L_1=20$ and for the 2nd order method $L_2 = 50000$.
\label{figure:mforsaken}}
\end{figure}
\subsection{With the competitive operator $F_{\alpha}$}

In this section we discuss simple extensions to our method that can be used to solve a larger class of saddle point problems. 

While \citet{diakonikolas2021efficient} consider the \textit{weak}-MVI condition for $F = \begin{bmatrix}
\nabla_x f\\
-\nabla_y f 
\end{bmatrix}$ our results are for any general field $F$. Making use of this versatility we show that our method can be extended to include the parameterized competitive field $g_{\alpha}$ introduced in \citet{vyas2022competitive}. This field is important as following it allows us to obtain small operator norm for a different generalisation of the MVI condition, $\alpha$-MVI \citet{vyas2022competitive}. For the said field, $$F_{\alpha} =  \begin{bmatrix}
I & \alpha \nabla_{xy}f\\
-\alpha \nabla_{yx}f & I
\end{bmatrix}^{-1}\begin{bmatrix}
\nabla_x f\\
-\nabla_y f 
\end{bmatrix} $$ the \textit{weak}-MVI assumption~\eqref{assmpt:balanced} for $p=1$ contains the $\alpha$-MVI class. Note that if a field $F$ satisfies the $\alpha$-MVI condition, the corresponding field $F_{\alpha}$ satisfies the MVI condition.

We further note that the exact stationary points of $F_{\alpha}$ are the same as that of $F$. Thus an additional requirement of at least one SVI solution with $F_{\alpha}$ as the operator is satisfied if a solution to SVI with $F$ as the operator exists.

 Using $F_{\alpha}$ as the field in question, we illustrate the first-order (which coincides with \textsl{oCGO}) and second order version of our algorithm on some examples satisfying the $\alpha$-MVI condition and show that our method performs faster than \textsl{oCGO}. Note that the first-order method is somewhere between first and second order while the second-order method is between second and third-order order in terms of the order of information of the gradient oracles. The second-order method can thus be thought of as a higher order \textsl{oCGO}. 
\begin{example}
We first consider the problem  $$\min_{x\in\cx}\max_{y \in \cy} x^2y$$
 
Note that while all points on the $y$-axis are stationary points of the saddle point problem generated by $f(x,y)=x^2y$, there is only one global Nash equilibrium at the origin. However, when \hoeg uses $F$ as the operator the iterates converge to different points on $x=0$ as can be seen in Figure~\ref{figure:x^2ya}. Using $F_{\alpha}$ as the operator mitigates this issue of convergence to non Nash stationary points and as $\alpha$ increases, iterates from all the different initialization converge to the origin which is the Nash equilibrium, Figure~\ref{figure:x^2yb}. Theoretically mapping out the relation between the nature of stationary points and the operator used in our method, remains an open direction of research. 
\end{example}

\begin{figure}[H]
\centering
\subcaptionbox{$F$\label{figure:x^2ya}}
[.48 \textwidth]{\includegraphics[width=1\linewidth]{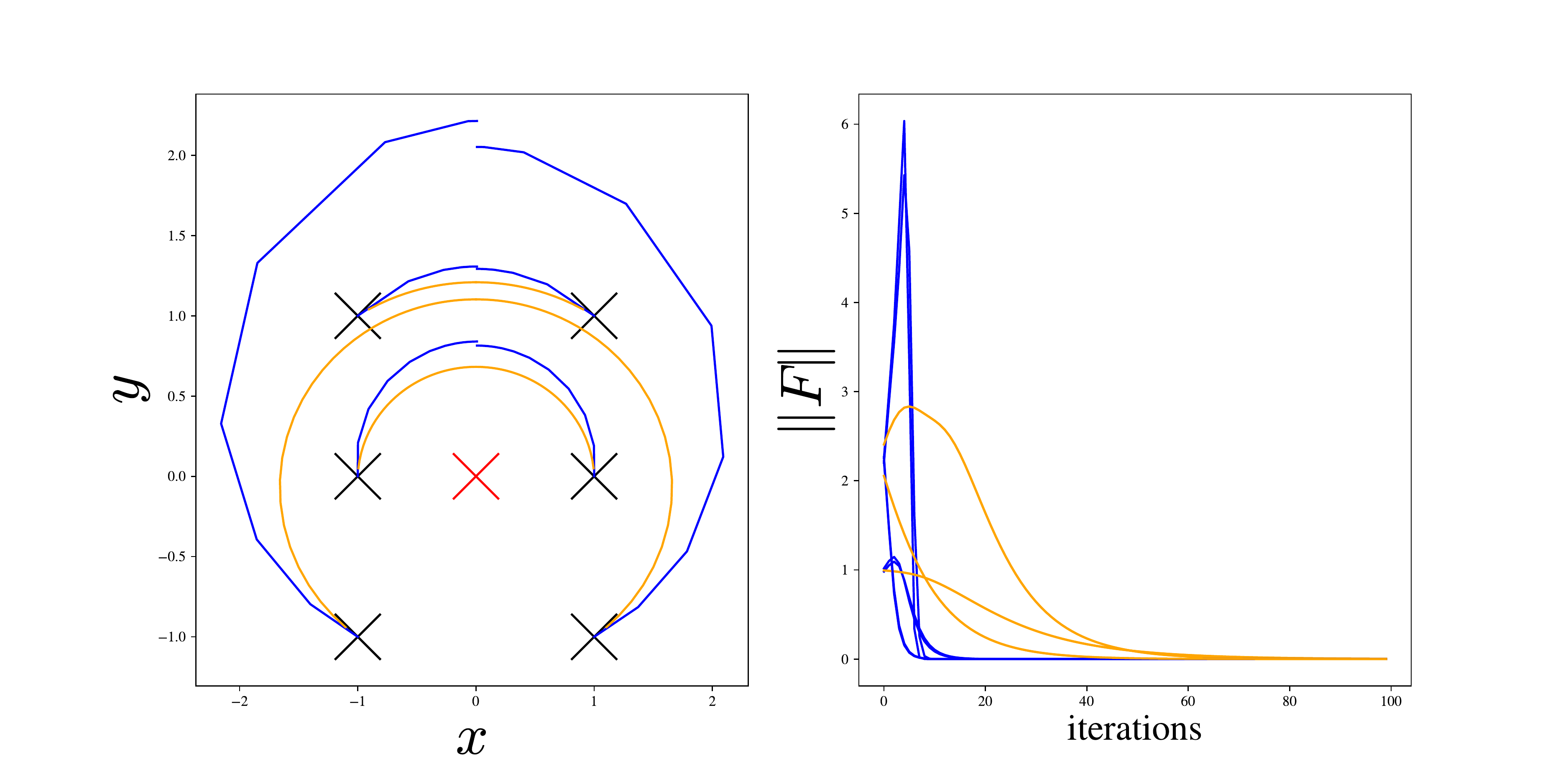}}
\subcaptionbox{$F_{\alpha};$ $\alpha=10$\label{figure:x^2yb}}
[.48 \textwidth]{\includegraphics[width=1\linewidth]{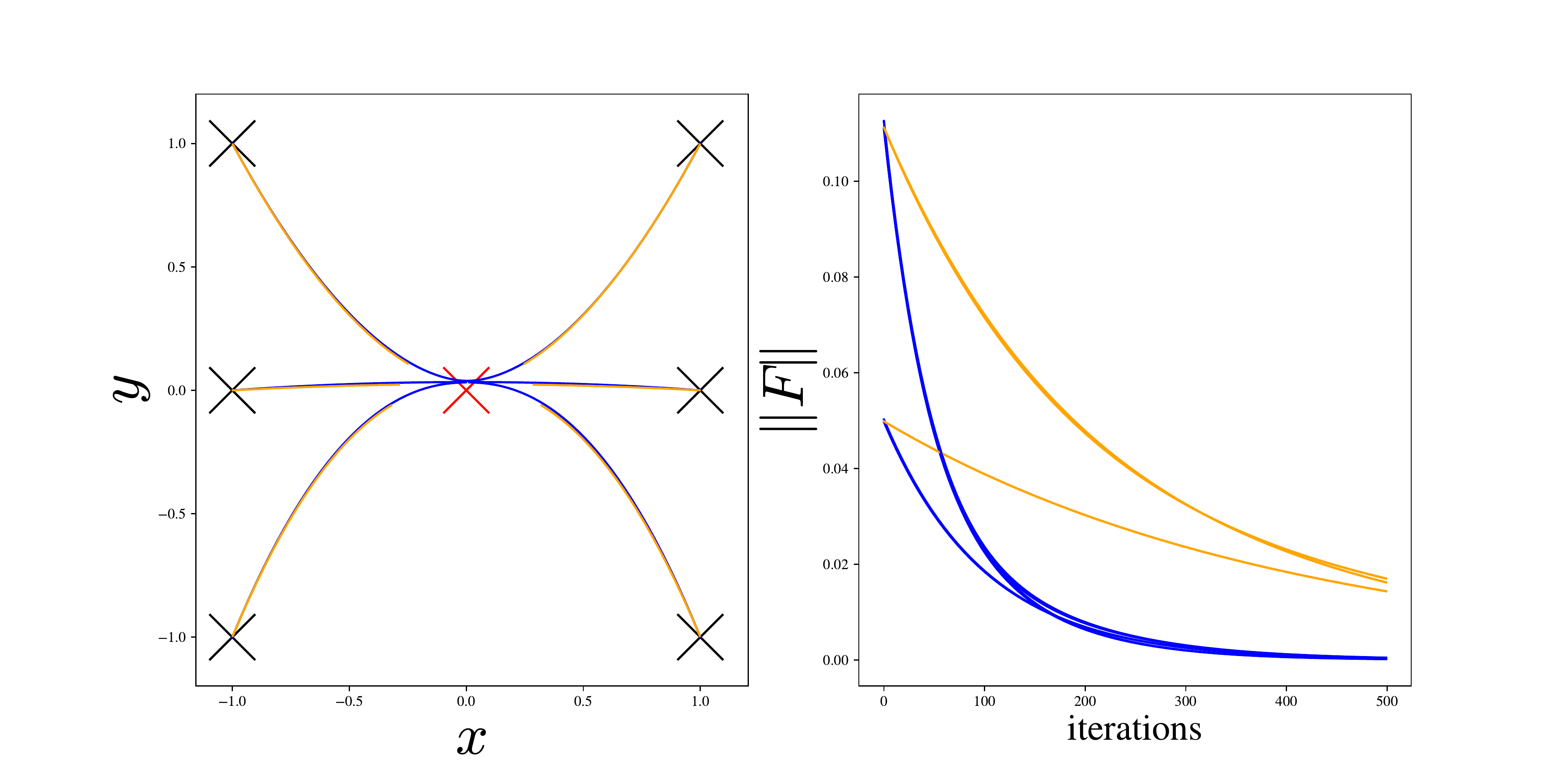}}
\caption{1st and 2nd Order method with $F_{\alpha}$ and $F$. For the 1st order method we use $L_1=20$ and for the 2nd order method $L_2 = 500$ as the Lipschitz constants for the operator}\label{figure:x^2y}
\end{figure}

\begin{example}
\begin{equation}\label{example:forsaken}\tag{Forsaken}
    \min_{|x|\leq \frac{3}{2}} \max_{|y|\leq \frac{3}{2}} f(x,y) = x(y-0.45)+h(x)-h(y)
\end{equation}
where $h(t) = \frac{t^2}{4}-\frac{t^4}{2}+\frac{t^6}{6}$.
\end{example}

We empirically study the Forsaken example in \citet{pethick2022escaping} and show that while using the gradient field $F=(\nabla_x f,-\nabla_y f)$ results in oscillation of both the 1st and 2nd order methods as can be seen in Figure \ref{figure:forsaken}, using $F_{\alpha}$ allows us to converge to the only stationary point of the function, $z^* = (0.0780,0.4119)$. This happens since even though the stationary point $z^*$ does not satisfy the $weak$-MVI condition for $F$, it satisfies (as we numerically verify) the \textit{weak}-MVI (and MVI) condition for $F_{\alpha}$ , $\alpha\geq 2$.

\begin{figure}[H]
\centering
\subcaptionbox{$F$\label{figure:forsakena}}
[.48 \textwidth]{\includegraphics[width=1\linewidth]{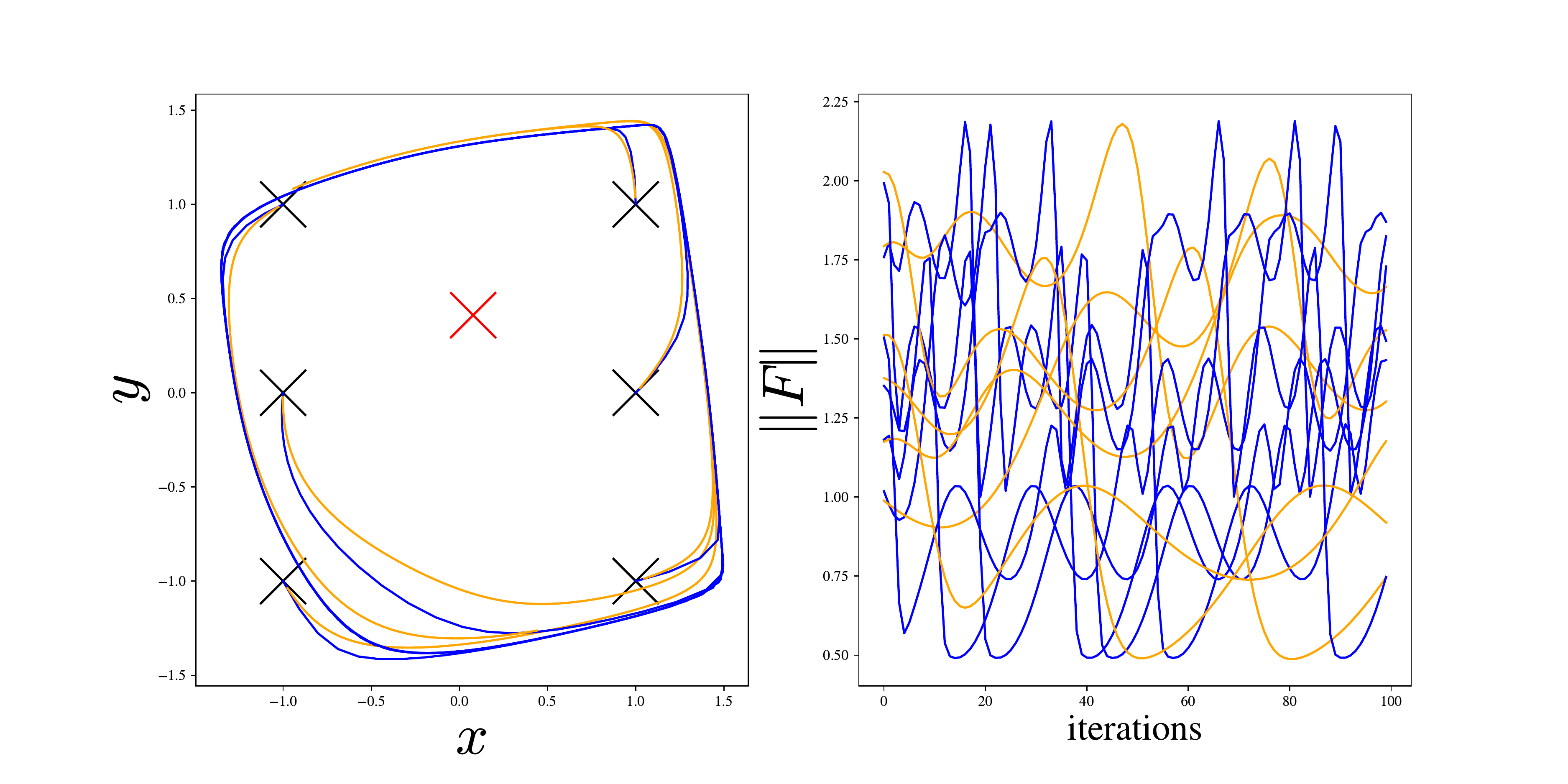}}
\subcaptionbox{$F_{\alpha};$ $\alpha=10$\label{figure:forsakenb}}
[.48 \textwidth]{\includegraphics[width=1\linewidth]{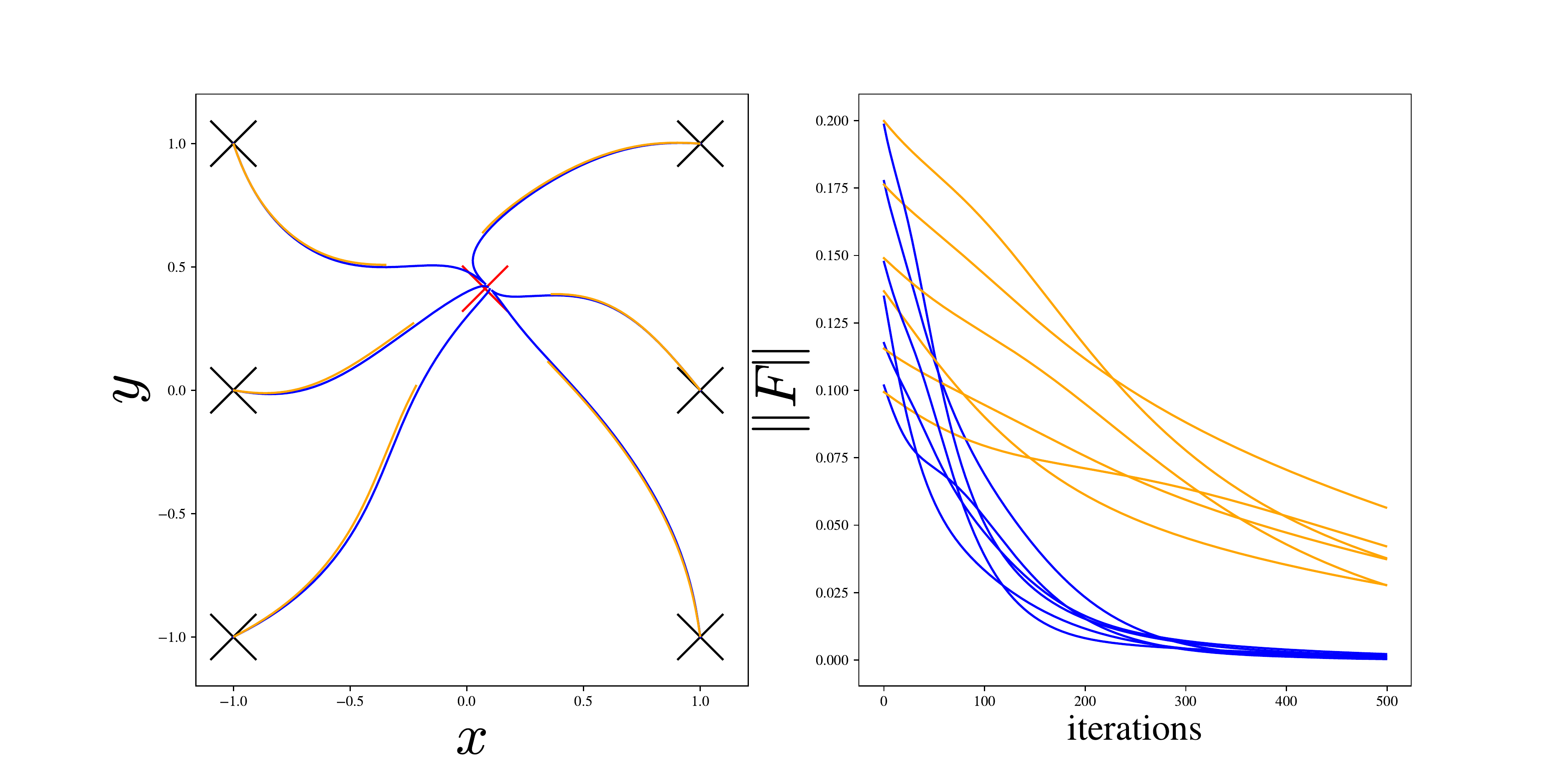}}
\caption{First and second-order method with $F$ and $F_{\alpha}$ on \eqref{example:forsaken}. While the algorithm using $F$ cycles, using $F_{\alpha}$ the algrithm converges to a stationary point.}\label{figure:forsaken}
\end{figure}
\section{Conclusion}
We propose higher-order methods for min-max problems satisfying a certain \textit{weak}-MVI condition \citep{diakonikolas2021efficient}, whereby the $p^{th}$-order instance of our algorithm finds approximate stationary points at a rate of $O(\frac{1}{\epsilon^\frac{p}{2}})$. We further establish that, in the continuous-time limit, \hoeg gives rise to the dynamical system \ref{sys:DE} \citep{lin2022continuous}, and the system obtains an analogous convergence rate as in the discrete-time setting. Finally we illustrate the performance of our algorithm experimentally for $p=1$ and $p=2$ using both the standard min-max operator $F=(\nabla_x f,-\nabla_y f)$ and the operator introduced by \cite{vyas2022competitive}, $F_{\alpha}$, demonstrating the potential advantage of our algorithm compared to first-order methods.

\bibliography{higherweak}

\begin{thebibliography}{38}
\providecommand{\natexlab}[1]{#1}
\providecommand{\url}[1]{\texttt{#1}}
\expandafter\ifx\csname urlstyle\endcsname\relax
  \providecommand{\doi}[1]{doi: #1}\else
  \providecommand{\doi}{doi: \begingroup \urlstyle{rm}\Url}\fi

\bibitem[Adil et~al.(2022)Adil, Bullins, Jambulapati, and
  Sachdeva]{adil2022optimal}
D.~Adil, B.~Bullins, A.~Jambulapati, and S.~Sachdeva.
\newblock Optimal methods for higher-order smooth monotone variational
  inequalities.
\newblock \emph{arXiv preprint arXiv:2205.06167}, 2022.

\bibitem[Agarwal et~al.(2017)Agarwal, Allen-Zhu, Bullins, Hazan, and
  Ma]{agarwal2017finding}
N.~Agarwal, Z.~Allen-Zhu, B.~Bullins, E.~Hazan, and T.~Ma.
\newblock Finding approximate local minima faster than gradient descent.
\newblock In \emph{Proceedings of the 49th Annual ACM SIGACT Symposium on
  Theory of Computing}, pages 1195--1199, 2017.

\bibitem[Ben-Tal et~al.(2009)Ben-Tal, El~Ghaoui, and Nemirovski]{ben2009robust}
A.~Ben-Tal, L.~El~Ghaoui, and A.~Nemirovski.
\newblock \emph{Robust optimization}, volume~28.
\newblock Princeton University Press, 2009.

\bibitem[Birgin et~al.(2017)Birgin, Gardenghi, Mart{\'\i}nez, Santos, and
  Toint]{birgin2017worst}
E.~G. Birgin, J.~Gardenghi, J.~M. Mart{\'\i}nez, S.~A. Santos, and P.~L. Toint.
\newblock Worst-case evaluation complexity for unconstrained nonlinear
  optimization using high-order regularized models.
\newblock \emph{Mathematical Programming}, 163:\penalty0 359--368, 2017.

\bibitem[Bullins and Lai(2022)]{bullins2022higher}
B.~Bullins and K.~A. Lai.
\newblock Higher-order methods for convex-concave min-max optimization and
  monotone variational inequalities.
\newblock \emph{SIAM Journal on Optimization}, 32\penalty0 (3):\penalty0
  2208--2229, 2022.

\bibitem[Carmon et~al.(2017)Carmon, Duchi, Hinder, and
  Sidford]{carmon2017convex}
Y.~Carmon, J.~C. Duchi, O.~Hinder, and A.~Sidford.
\newblock “convex until proven guilty”: Dimension-free acceleration of
  gradient descent on non-convex functions.
\newblock In \emph{International conference on machine learning}, pages
  654--663. PMLR, 2017.

\bibitem[Carmon et~al.(2018)Carmon, Duchi, Hinder, and
  Sidford]{carmon2018accelerated}
Y.~Carmon, J.~C. Duchi, O.~Hinder, and A.~Sidford.
\newblock Accelerated methods for nonconvex optimization.
\newblock \emph{SIAM Journal on Optimization}, 28\penalty0 (2):\penalty0
  1751--1772, 2018.

\bibitem[Carmon et~al.(2020)Carmon, Duchi, Hinder, and
  Sidford]{carmon2020lower}
Y.~Carmon, J.~C. Duchi, O.~Hinder, and A.~Sidford.
\newblock Lower bounds for finding stationary points i.
\newblock \emph{Mathematical Programming}, 184\penalty0 (1-2):\penalty0
  71--120, 2020.

\bibitem[Carmon et~al.(2021)Carmon, Duchi, Hinder, and
  Sidford]{carmon2021lower}
Y.~Carmon, J.~C. Duchi, O.~Hinder, and A.~Sidford.
\newblock Lower bounds for finding stationary points ii: first-order methods.
\newblock \emph{Mathematical Programming}, 185\penalty0 (1-2):\penalty0
  315--355, 2021.

\bibitem[Daskalakis et~al.(2021)Daskalakis, Skoulakis, and
  Zampetakis]{daskalakis2021complexity}
C.~Daskalakis, S.~Skoulakis, and M.~Zampetakis.
\newblock The complexity of constrained min-max optimization.
\newblock In \emph{Proceedings of the 53rd Annual ACM SIGACT Symposium on
  Theory of Computing}, pages 1466--1478, 2021.

\bibitem[Diakonikolas(2020)]{diakonikolas2020halpern}
J.~Diakonikolas.
\newblock Halpern iteration for near-optimal and parameter-free monotone
  inclusion and strong solutions to variational inequalities.
\newblock In \emph{Conference on Learning Theory}, pages 1428--1451. PMLR,
  2020.

\bibitem[Diakonikolas et~al.(2021)Diakonikolas, Daskalakis, and
  Jordan]{diakonikolas2021efficient}
J.~Diakonikolas, C.~Daskalakis, and M.~I. Jordan.
\newblock Efficient methods for structured nonconvex-nonconcave min-max
  optimization.
\newblock In \emph{International Conference on Artificial Intelligence and
  Statistics}, pages 2746--2754. PMLR, 2021.

\bibitem[Goodfellow et~al.(2020)Goodfellow, Pouget-Abadie, Mirza, Xu,
  Warde-Farley, Ozair, Courville, and Bengio]{goodfellow2020generative}
I.~Goodfellow, J.~Pouget-Abadie, M.~Mirza, B.~Xu, D.~Warde-Farley, S.~Ozair,
  A.~Courville, and Y.~Bengio.
\newblock Generative adversarial networks.
\newblock \emph{Communications of the ACM}, 63\penalty0 (11):\penalty0
  139--144, 2020.

\bibitem[Jiang and Mokhtari(2022)]{jiang2022generalized}
R.~Jiang and A.~Mokhtari.
\newblock Generalized optimistic methods for convex-concave saddle point
  problems.
\newblock \emph{arXiv preprint arXiv:2202.09674}, 2022.

\bibitem[Jin et~al.(2018)Jin, Netrapalli, and Jordan]{jin2018accelerated}
C.~Jin, P.~Netrapalli, and M.~I. Jordan.
\newblock Accelerated gradient descent escapes saddle points faster than
  gradient descent.
\newblock In \emph{Conference On Learning Theory}, pages 1042--1085. PMLR,
  2018.

\bibitem[Korpelevich(1976)]{korpelevich1976extragradient}
G.~M. Korpelevich.
\newblock The extragradient method for finding saddle points and other
  problems.
\newblock \emph{Matecon}, 12:\penalty0 747--756, 1976.

\bibitem[Latz(2021)]{latz2021analysis}
J.~Latz.
\newblock Analysis of stochastic gradient descent in continuous time.
\newblock \emph{Statistics and Computing}, 31\penalty0 (4):\penalty0 39, 2021.

\bibitem[Lee and Kim(2021)]{lee2021fast}
S.~Lee and D.~Kim.
\newblock Fast extra gradient methods for smooth structured
  nonconvex-nonconcave minimax problems.
\newblock \emph{Advances in Neural Information Processing Systems},
  34:\penalty0 22588--22600, 2021.

\bibitem[Li and Lin(2022)]{li2022restarted}
H.~Li and Z.~Lin.
\newblock Restarted nonconvex accelerated gradient descent: No more
  polylogarithmic factor in the o(exp(e)(-7/4)) complexity.
\newblock In \emph{International Conference on Machine Learning}, pages
  12901--12916. PMLR, 2022.

\bibitem[Lin and Jordan(2022{\natexlab{a}})]{lin2022continuous}
T.~Lin and M.~Jordan.
\newblock A continuous-time perspective on monotone equation problems.
\newblock \emph{arXiv preprint arXiv:2206.04770}, 2022{\natexlab{a}}.

\bibitem[Lin and Jordan(2022{\natexlab{b}})]{lin2022perseus}
T.~Lin and M.~Jordan.
\newblock Perseus: A simple high-order regularization method for variational
  inequalities.
\newblock \emph{arXiv preprint arXiv:2205.03202}, 2022{\natexlab{b}}.

\bibitem[Madry et~al.(2018)Madry, Makelov, Schmidt, Tsipras, and
  Vladu]{madry2018towards}
A.~Madry, A.~Makelov, L.~Schmidt, D.~Tsipras, and A.~Vladu.
\newblock Towards deep learning models resistant to adversarial attacks.
\newblock In \emph{International Conference on Learning Representations}, 2018.

\bibitem[Mazumdar et~al.(2020)Mazumdar, Ratliff, and
  Sastry]{mazumdar2020gradient}
E.~Mazumdar, L.~J. Ratliff, and S.~S. Sastry.
\newblock On gradient-based learning in continuous games.
\newblock \emph{SIAM Journal on Mathematics of Data Science}, 2\penalty0
  (1):\penalty0 103--131, 2020.

\bibitem[Mertikopoulos et~al.(2018)Mertikopoulos, Lecouat, Zenati, Foo,
  Chandrasekhar, and Piliouras]{mertikopoulos2018optimistic}
P.~Mertikopoulos, B.~Lecouat, H.~Zenati, C.-S. Foo, V.~Chandrasekhar, and
  G.~Piliouras.
\newblock Optimistic mirror descent in saddle-point problems: Going the extra
  (gradient) mile.
\newblock \emph{arXiv preprint arXiv:1807.02629}, 2018.

\bibitem[Nemirovski(2004)]{nemirovski2004prox}
A.~Nemirovski.
\newblock Prox-method with rate of convergence o (1/t) for variational
  inequalities with lipschitz continuous monotone operators and smooth
  convex-concave saddle point problems.
\newblock \emph{SIAM Journal on Optimization}, 15\penalty0 (1):\penalty0
  229--251, 2004.

\bibitem[Nesterov(2007)]{nesterov2007dual}
Y.~Nesterov.
\newblock Dual extrapolation and its applications to solving variational
  inequalities and related problems.
\newblock \emph{Mathematical Programming}, 109\penalty0 (2-3):\penalty0
  319--344, 2007.

\bibitem[Nesterov and Polyak(2006)]{nesterov2006cubic}
Y.~Nesterov and B.~T. Polyak.
\newblock Cubic regularization of newton method and its global performance.
\newblock \emph{Mathematical Programming}, 108\penalty0 (1):\penalty0 177--205,
  2006.

\bibitem[Ouyang and Xu(2021)]{ouyang2021lower}
Y.~Ouyang and Y.~Xu.
\newblock Lower complexity bounds of first-order methods for convex-concave
  bilinear saddle-point problems.
\newblock \emph{Mathematical Programming}, 185\penalty0 (1-2):\penalty0 1--35,
  2021.

\bibitem[Pethick et~al.(2022)Pethick, Patrinos, Fercoq, Cevher{\aa},
  et~al.]{pethick2022escaping}
T.~Pethick, P.~Patrinos, O.~Fercoq, V.~Cevher{\aa}, et~al.
\newblock Escaping limit cycles: Global convergence for constrained
  nonconvex-nonconcave minimax problems.
\newblock In \emph{International Conference on Learning Representations}, 2022.

\bibitem[Sch{\"a}fer and Anandkumar(2019)]{schafer2019competitive}
F.~Sch{\"a}fer and A.~Anandkumar.
\newblock Competitive gradient descent.
\newblock \emph{Advances in Neural Information Processing Systems}, 32, 2019.

\bibitem[Shi et~al.(2021)Shi, Du, Jordan, and Su]{shi2021understanding}
B.~Shi, S.~S. Du, M.~I. Jordan, and W.~J. Su.
\newblock Understanding the acceleration phenomenon via high-resolution
  differential equations.
\newblock \emph{Mathematical Programming}, pages 1--70, 2021.

\bibitem[Song et~al.(2020)Song, Zhou, Zhou, Jiang, and Ma]{song2020optimistic}
C.~Song, Z.~Zhou, Y.~Zhou, Y.~Jiang, and Y.~Ma.
\newblock Optimistic dual extrapolation for coherent non-monotone variational
  inequalities.
\newblock \emph{Advances in Neural Information Processing Systems},
  33:\penalty0 14303--14314, 2020.

\bibitem[Tseng(2008)]{tseng2008accelerated}
P.~Tseng.
\newblock Accelerated proximal gradient methods for convex optimization.
\newblock Technical report, University of Washington, Seattle, 2008.

\bibitem[Vyas and Azizzadenesheli(2022)]{vyas2022competitive}
A.~Vyas and K.~Azizzadenesheli.
\newblock Competitive gradient optimization.
\newblock \emph{arXiv preprint arXiv:2205.14232}, 2022.

\bibitem[Wibisono et~al.(2016)Wibisono, Wilson, and
  Jordan]{wibisono2016variational}
A.~Wibisono, A.~C. Wilson, and M.~I. Jordan.
\newblock A variational perspective on accelerated methods in optimization.
\newblock \emph{proceedings of the National Academy of Sciences}, 113\penalty0
  (47):\penalty0 E7351--E7358, 2016.

\bibitem[Wilson et~al.(2016)Wilson, Recht, and Jordan]{wilson2016lyapunov}
A.~C. Wilson, B.~Recht, and M.~I. Jordan.
\newblock A lyapunov analysis of momentum methods in optimization.
\newblock \emph{arXiv preprint arXiv:1611.02635}, 2016.

\bibitem[Yoon and Ryu(2021)]{yoon2021accelerated}
T.~Yoon and E.~K. Ryu.
\newblock Accelerated algorithms for smooth convex-concave minimax problems
  with o (1/k\^{} 2) rate on squared gradient norm.
\newblock In \emph{International Conference on Machine Learning}, pages
  12098--12109. PMLR, 2021.

\bibitem[Zhang et~al.(2019)Zhang, Yu, Jiao, Xing, El~Ghaoui, and
  Jordan]{zhang2019theoretically}
H.~Zhang, Y.~Yu, J.~Jiao, E.~Xing, L.~El~Ghaoui, and M.~Jordan.
\newblock Theoretically principled trade-off between robustness and accuracy.
\newblock In \emph{International conference on machine learning}, pages
  7472--7482. PMLR, 2019.

\end{thebibliography}
\bibliographystyle{abbrvnat}
\newpage
\appendix
\section{Supporting Lemmas}\label{subsec:supportive-dt}

We now present the supportive lemmas used to prove Lemma \ref{lemma:supportive}. 

Let $\omega(z_a,z_b)$ denote the Bregman divergence of $h$,
i.e.,
\begin{equation*}
\omega(z_a,z_b)   =  h(z_a) - h(z_b) - \langle \nabla h(z_b),z_a-z_b\rangle
\end{equation*}

\begin{lemma}[Three Point Property]\label{lem:3point}
Let $\omega(z,z_b)$ denote the Bregman divergence of a function $h$. The \textit{three point property} states, for any $z_a,z_b,z_c \in \cz$,
\[
\langle \nabla h(z_b) - \nabla h(z_c), z_a-z_c \rangle = \omega(z_a,z_c) + \omega(z_c,z_b) - \omega(z_a,z_b).
\]
\end{lemma}

\begin{lemma}[\cite{tseng2008accelerated}]\label{lem:Tseng}
 Let $\phi$ be a convex function, let $z_a\in \cz$, and let
\[
z_c = \arg\min_{z'\in \cz}\{\phi(z') + \omega(z',z_a)\}.
\]
Then, for all $z_b\in \cz$, we have,
$ \phi(z_b) + \omega(z_b,z_a) \geq \phi(z_c) + \omega(z_c,z_a) +
\omega(z_b,z_c).$
\end{lemma}

Finally, for the sake of completion, we provide here the proof of the key lemma from \cite{adil2022optimal} (Lemma \ref{lemma:supportive}).

\subsection{Proof of Lemma \ref{lemma:supportive}}

\begin{proof}
For any $k$ and any $z \in \calZ$, we first apply Lemma~\ref{lem:Tseng} with $\phi(z) = \lambda_k \frac{p!}{L_p} \langle F(z_{k+\frac{1}{2}}),z - z_k\rangle$, which gives us
\begin{equation}\label{eq:Tsengsec3}
\lambda_k \frac{p!}{L_p} \langle F(z_{k+\frac{1}{2}}),z_{k+1} -z\rangle \leq \omega(z,z_k) - \omega(z,z_{k+1}) - \omega(z_{k+1},z_k). 
\end{equation}
Additionally, the guarantee of assumption~\ref{assmpt:smooth} with $z = z_{k+1}$ yields
\begin{equation}
\label{eqn:MPstep1sec3}
\left\langle \tau_{p-1}(z_{k+\frac{1}{2}}; z_k) , z_{k+\frac{1}{2}} - z_{k+1} \right\rangle  \leq \frac{2 L_p}{p!} \omega(z_{k+\frac{1}{2}}, z_k)^{\frac{p-1}{2}} \left\langle \nabla h(z_k) - \nabla h(z_{k+\frac{1}{2}}),  z_{k+\frac{1}{2}} - z_{k+1}  \right\rangle.
\end{equation}
Applying the Bregmann three point property (Lemma~\ref{lem:3point}) and the definition of $\lambda_k$ to Equation~\ref{eqn:MPstep1sec3}, we have
\begin{equation}
\label{eqn:Tseng2sec3}
\lambda_k \frac{p!}{L_p}  \left\langle \tau_{p-1}(z_{k+\frac{1}{2}}; z_k) , z_{k+\frac{1}{2}} - z_{k+1} \right\rangle  \leq \omega(z_{k+1}, z_k) - \omega(z_{k+1}, z_{k+\frac{1}{2}}) - \omega(z_{k+\frac{1}{2}}, z_k).
\end{equation}

Summing Eqs.~\eqref{eq:Tsengsec3} and \eqref{eqn:Tseng2sec3}, we obtain 
\begin{align}
  \label{eqn:one_kter_1sec3}
&\lambda_k \frac{p!}{L_p} \left( \langle F(z_{k+\frac{1}{2}}),z_{i+\frac{1}{2} } -z\rangle + \left\langle   \tau_{p-1}(z_{k+\frac{1}{2}}; z_k) - F(z_{i + \frac{1}{2} }) , z_{k+\frac{1}{2}} - z_{k+1} \right\rangle \right) \nonumber \\
&\leq \omega(z,z_k) - \omega(z,z_{k+1}) - \omega(z_{k+1}, z_{k+\frac{1}{2}}) - \omega(z_{k+\frac{1}{2}}, z_k).
\end{align}

Now, we obtain 
\begin{align*}
\lambda_k &\frac{p!}{L_p} \left\langle   \tau_{p-1}(z_{k+\frac{1}{2}}; z_k) - F(z_{i + \frac{1}{2} }) , z_{k+\frac{1}{2}} - z_{k+1} \right\rangle \\
&\substack{(i) \\ \geq} - \lambda_k \frac{p!}{L_p} \norm{\tau_{p-1}(z_{k+\frac{1}{2}}; z_k) - F(z_{i + \frac{1}{2} }) }_* \norm{ z_{k+\frac{1}{2}} - z_{k+1}} \\
&\substack{(ii) \\ \geq} - \lambda_k \norm{z_{i+ \frac{1}{2}} -  z_k}^p \norm{ z_{k+\frac{1}{2}} - z_{k+1}} \\ 
&\substack{(iii) \\ \geq} -\frac{1}{2} \omega(z_{i + \frac{1}{2}}, z_k)^{-\frac{p-1}{2}} \omega(z_{i + \frac{1}{2}}, z_k)^{\frac{p}{2}} \omega(z_{k+1}, z_{k+\frac{1}{2}})^{\frac{1}{2}} \\ 
&= - \frac{1}{2} \omega(z_{i + \frac{1}{2}}, z_k)^{\frac{1}{2}} \omega(z_{k+1}, z_{k+\frac{1}{2}})^{\frac{1}{2}} \\
&\substack{(iv) \\ \geq} - \frac{1}{16}  \omega(z_{i + \frac{1}{2}}, z_k) - \omega(z_{k+1}, z_{k+\frac{1}{2}})
\end{align*}

Here, $(i)$ used Holder's inequality, $(ii)$ used  assumption~\ref{assmpt:smooth}, $(iii)$ used the $1$-strong convexity of $\omega$, and $(iv)$ used the inequality $\sqrt{xy} \leq 2x + \frac{1}{8} y$ for $x,y \geq 0$. Combining with Eq.~\eqref{eqn:one_kter_1sec3} and rearranging yields
\begin{equation*}
\lambda_k \frac{p!}{L_p} \langle F(z_{k+\frac{1}{2}}),z_{i+\frac{1}{2} } -z\rangle \leq \omega(z,z_k) - \omega(z,z_{k+1}) - \frac{15}{16} \omega(z_{k+\frac{1}{2}}, z_k). 
\end{equation*}
We observe that by the choice of $\lambda_k$ in Algorithm \ref{alg:mainalg}, $\omega(z_{i + \frac{1}{2}}, z_k) = \left(2 \lambda_k\right)^{- \frac{2}{p-1}}$. Applying this fact and summing over all iterations $i$ yields 
\[
\sum_{i=0}^K  \lambda_k \frac{p!}{L_p} \langle F(z_{k+\frac{1}{2}}),z_{i+\frac{1}{2} } -z\rangle \leq \omega(z, z_0) - \frac{15}{16} \sum_{i=0}^K \left(2 \lambda_k\right)^{- \frac{2}{p-1}}
\]

Finally substituting $z=z^*$ and setting the potential function $h(x) = \|x\|^2$ in $\omega$ gives us the statement of the lemma.
\end{proof}

\section{Results for de-coupled weak-MVI condition and $F_{\alpha}$} 

In this section we discuss the rates obtained when higher order methods are applied to the function which satisfy the original weak-MVI condition

\begin{assumption}[$q^{th}$ Weak \textsc{mvi}]
There exists $z^* \in \cz^*$ such that:
\begin{equation}\label{assmpt:imbalanced}\tag{\textsc{a}$_3$}
    (\forall z \in \rr^d):\quad 
    \innp{F(z), z - z^*} \geq -\frac{\rho}{2} \|F(z)\|^q,
\end{equation}
for some parameter $\rho$.
\end{assumption}

\begin{theorem}\label{theorem:imbalanced}
For $F$ satisfying assumption~\eqref{assmpt:imbalanced} with 
 $\rho\leq\frac{15}{16D^q}\frac{p!D}{L_p}^{\frac{p+1}{p}}$ and $D = \max_k L_1\|z_k-z^*\|$, running a $p^{th}$-order instance of our Algorithm~\eqref{alg:mainalg} we have,

\begin{itemize}
    \item[(i)] for all $k\geq 1:$ 
    $$
        \frac{1}{k+1}\sum_{i=0}^k \|F(z_{k+\frac{1}{2}})\|^\frac{2}{p} \leq \frac{c \|z_0 - z^*\|^2}{k+1}.
    $$
    In particular, we have that 
    \begin{align*}
        \min_{0\leq i\leq k} \|F(z_{k+\frac{1}{2}})\|^\frac{2}{p} \leq \frac{c\|z_0 - z^*\|^2}{k+1}\rightarrow
        \min_{0\leq i\leq k} \|F(z_k)\|^2 \leq \frac{c_1}{(k+1)^p}
    \end{align*}

    and 
    $$
        \ee_{i \sim \mathrm{Unif}\{0, \dots, k\}}\big[\|F(z_k)\|^2\big] \leq \frac{c \|z_0 - z^*\|^2}{k+1}^p \leq \frac{C}{(k+1)^p},
    $$
    where $i \sim \mathrm{Unif}\{0, \dots, k\}$ denotes an index $i$ chosen uniformly at random from the set $\{0, \dots, k\}.$

    \item[(ii)] for $K = O(\frac{1}{\epsilon^\frac{p}{2}})$ number of iterations the output $z_{out}$ of Algorithm \ref{alg:mainalg} is an $\epsilon$-approximate stationary point i.e., it satisfies $\|F(z_{out})\| \leq \epsilon$
\end{itemize}
\end{theorem}
\begin{proof}

From Lemma 3.2 \citep{adil2022optimal}, for Algorithm~\eqref{alg:mainalg},  we have: 

\begin{align}\label{eq:maind}
    \sum_{i=0}^K \lambda_k \frac{p!}{L_p}\langle F(z_{k+\frac{1}{2}}),z_{k+\frac{1}{2}}-z^* \rangle \leq\|z_0 - z^*\|^2-\frac{15}{16}\sum_{i=0}^K\|z_k-z_{k+\frac{1}{2}}\|^2
\end{align}

Adding $\sum_{k=1}^K\rho\lambda_k\|F(z_k)\|^q$ to both sides of Eq. \eqref{eq:maind}, and noting from Eq.~\eqref{lemma:upper-b} that,

\begin{align*}
\rho\lambda_k\|F(z_{k+\frac{1}{2}})\|^q = \|F(z_{k+\frac{1}{2}})\|^{(q-\frac{p+1}{p})} \rho\frac{\|F(z_{k+\frac{1}{2}})\|^{\frac{p+1}{p}}}{2\|z_k-z_{k+\frac{1}{2}}\|^{p-1}} \leq \rho (\frac{L_p}{p!})^{\frac{p+1}{p}}\|z_k-z_{k+\frac{1}{2}}\|^2 \|F(z_{k+\frac{1}{2}})\|^{(q-\frac{p+1}{p})}
\end{align*}

 Further substituting $p=1$ in and choosing $x,y$ appropriately from \eqref{assmpt:smooth} we obtain:
    
$$\|F(z_{k+\frac{1}{2}})\| \leq L_1\|z_{k+\frac{1}{2}}-z^*\|$$

further letting $D = \max_k L_1\|z_{k+\frac{1}{2}}-z^*\|$ we have,

\begin{align*}
\rho (\frac{L_p}{p!})^{\frac{p+1}{p}}\|z_k-z_{k+\frac{1}{2}}\|^2 \|F(z_{k+\frac{1}{2}})\|^{(q-\frac{p+1}{p})} \leq \rho D^q(\frac{L_p}{p!D})^{\frac{p+1}{p}}\|z_k-z_{k+\frac{1}{2}}\|^2\nonumber
\end{align*}.

Thus,
\begin{align}\label{main3}
    \sum_{i=0}^K \lambda_k (\frac{p!}{L_p}\langle F(z_{k+\frac{1}{2}}),z_{k+\frac{1}{2}}-z^* \rangle + \rho\|F(z_{k+\frac{1}{2}})\|^q)  \leq &\|z^*-z_0\|^2\\-&(\frac{15}{16}-\rho D^q(\frac{L_p}{p!D})^{\frac{p+1}{p}})\sum_{i=0}^K\|z_k-z_{k+\frac{1}{2}}\|^2\nonumber
\end{align}

From assumption~\eqref{assmpt:imbalanced} we have that LHS in Eq.~\eqref{main3} is non-negative. Setting $c_2 = (\frac{15}{16}-\rho D^q(\frac{L_p}{p!D})^{\frac{p+1}{p}})$ and using Lemma~\ref{lemma:upper-b} we get,

\begin{equation*}
   c_2\sum_{i=0}^K\|F(z_{k+\frac{1}{2}})\|^{\frac{2}{p}} \leq c_2\sum_{i=0}^K\|z_k-z_{k+\frac{1}{2}}\|^2\leq \|z^*-z_0\|^2
\end{equation*}

And thus we obtain results analogous to Theorem~\ref{theorem:balanced}. Furthermore, setting $q=2$ gives us rates for the $p^{th}$-order versions of \hoeg for the weak-MVI condition in \citep{diakonikolas2021efficient}
\end{proof}

We now present the result on the convergence of $F_{\alpha}$ when used as the operator in our algorithm. As a corollary of Theorem \ref{theorem:balanced} we have,

\begin{corollary}
Let the competitive field $F_{\alpha}$ satisfy the smoothness and weak-MVI assumptions \eqref{assmpt:balanced}, \eqref{assmpt:smooth} then for algorithm \eqref{alg:mainalg} using $F_{\alpha}$ we have,
\begin{itemize}
    \item[(i)] for all $k\geq 1:$ 
    $$
        \frac{1}{K+1}\sum_{k=0}^K \|F_{\alpha}(z_{k+\frac{1}{2}})\|^\frac{2}{p} \leq \frac{ \|z_0 - z^*\|^2}{c_2(K+1)}.
    $$
    In particular, we have that 
    \begin{align*}
        \min_{0\leq k\leq K} \|F_{\alpha}(z_{k+\frac{1}{2}})\|^\frac{2}{p} \leq \frac{\|z_0 - z^*\|^2}{c_2(K+1)}\rightarrow \min_{0\leq k\leq K} \|F_{\alpha}(z_{k+\frac{1}{2}})\|^2 \leq \frac{C}{(K+1)^p} 
    \end{align*}

    where $C=\frac{\|z_0-z^*\|^{2p}}{c_2^p}$ and 
    \begin{align*}
        \ee_{k \sim \mathrm{Unif}\{0, \dots, K\}}\big[\|F_{\alpha}(z_{k+\frac{1}{2}})\|^2\big] \leq (\frac{ \|z_0 - z^*\|^2}{c_2(K+1)})^p \leq \frac{C}{(K+1)^p},
    \end{align*}
    
    where $k \sim \mathrm{Unif}\{0, \dots, K\}$ denotes an index $i$ chosen uniformly at random from the set $\{0, \dots, K\}.$

    \item[(ii)] for $K = O(\frac{1}{\epsilon^\frac{p}{2}})$ number of iterations the output $z_{out}$ of Algorithm \ref{alg:mainalg} is an $\epsilon$-approximate stationary point i.e., it satisfies $\|F(z_{out})\| \leq \epsilon$
\end{itemize}
\end{corollary}

\end{document}